\documentclass[a4paper,11pt]{article}
\usepackage[utf8]{inputenc}
\usepackage[T1]{fontenc}
\usepackage[french,english]{babel}
\usepackage[margin=2.8cm]{geometry}
\usepackage{mathpazo}
\usepackage[scaled=.95]{helvet}
\usepackage{courier}
\usepackage{amsmath,amsfonts,amssymb,amsthm}
\usepackage{graphicx}
\usepackage{natbib,url}
\usepackage{color}
\usepackage[onehalfspacing]{setspace}

\newcommand{\Rlogo}{\protect\includegraphics[height=1.8ex,keepaspectratio]{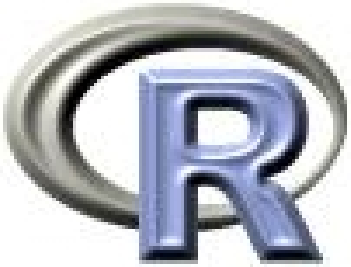}}
\newcommand{\egaldef}{\mathrel{:=}}
\newcommand*{\CF}{\mathcal{F}}
\newcommand*{\E}{\mathbb{E}}
\newcommand*{\EE}[1]{\E\left[#1\right]}
\newcommand*{\EFn}[1]{\E\left[#1\middle| \CF_n\right]}
\newcommand*{\PP}{\mbox{$\mathbb{P}$}}
\newcommand*{\prb}[1]{\PP\left[#1\right]}
\newcommand*{\R}{\mathbb{R}}
\newcommand*{\xR}{\mathbb{R}}
\newcommand*{\xN}{\mathbb{N}}
\newcommand*{\scal}[1]{\left\langle #1 \right\rangle} 
\newcommand*{\nrm}[1]{\left\| #1 \right\|}            
\newcommand*{\abs}[1]{\left| #1 \right|}              
\newcommand*{\Var}{\mathbf{Var}}                      

\DeclareMathOperator{\SpanOp}{span}
\newcommand*{\vect}[1]{\SpanOp(#1)}
\newcommand*{\ind}[1]{\mathbf{1}_{#1}}
\newcommand{\CL}{\mathcal{L}}

\newcommand{\idt}{\mathbf{I}_H}
\newcommand*{\stp}[1]{U_{#1}}

\newcommand{\Zbar}{\overline{Z}}
\newcommand{\Tbar}{\overline{T}}

\newcommand{\muDisc}{\mu_{d}}
\newcommand{\muDiff}{\mu_c}

\newcommand*{\lUn}{L^1}
\newcommand*{\lDeux}{L^2}

\newcommand{\lmin}{\lambda_{min}}
\newcommand{\cG}{c_\gamma}
\newcommand{\covLimite}{\Sigma}

\newcommand{\MM}{\foreignlanguage{french}{Médiamétrie}}

\newcommand{\Clang}{\textbf{C}}
\newcommand{\newM}{\widetilde{M}}

\newcommand{\limite}[2]{\xrightarrow[#1]{#2}}

\DeclareMathOperator{\cvp}{\limite{n \to \infty}{P}}
\DeclareMathOperator{\cvl}{\limite{n \to \infty}{\CL}}

 \newcommand*{\ball}{\mathcal{B}(0,A)}

\newtheorem{thm}{Theorem}[section]
\newtheorem{cor}[thm]{Corollary}
\newtheorem{lem}[thm]{Lemma}
\newtheorem{prop}[thm]{Proposition}

\newtheorem{rem}{Remark}

\begin{document}

\title{Efficient and fast estimation of the geometric median in Hilbert spaces with an averaged stochastic gradient algorithm.}

\author{Herv\'e \textsc{Cardot}, Peggy \textsc{C\'enac}, Pierre-Andr\'e \textsc{Zitt} \\ Institut de Math\'ematiques de Bourgogne, Universit\'e de Bourgogne, \\
9 Rue Alain Savary, 21078 Dijon, France \\
email: \{Herve.Cardot, Peggy.Cenac,  Pierre-Andre.Zitt\}@u-bourgogne.fr
} 
\maketitle

\begin{abstract}
With the progress of measurement apparatus and the development of automatic sensors it is not unusual anymore to get large samples of observations taking values in high dimension spaces such as functional spaces. In such large samples of high dimensional data, outlying curves may not be
uncommon and  even  a few individuals may corrupt simple statistical indicators such
as the mean trajectory.  We focus here on the  estimation of the geometric median
which is  a direct generalization  of the real median in metric spaces and  has nice robustness properties. The geometric median being defined as the minimizer of a simple convex 
functional that is differentiable  everywhere when the distribution has
no atom, it is possible to estimate it with  online gradient
algorithms. Such algorithms are very fast and can deal with large samples. Furthermore they also can be simply updated when the data arrive sequentially.
We state the  almost sure consistency and  the $\lDeux$ rates of convergence
of the stochastic gradient estimator as well as the asymptotic normality of
its averaged version. We get that the asymptotic distribution of the averaged
version of the algorithm is the same as the classic estimators which are based on the minimization of the empirical loss function. The performances of our averaged sequential estimator, both in terms of computation speed and accuracy of the estimations, are evaluated with a small simulation study. Our approach is also  illustrated on  a sample of more than 5000 individual television audiences measured every second over a period of 24 hours.
\end{abstract}


\section{Introduction}

With the progress of measurement apparatus, the development of automatic sensors and  the increasing storage performances of computers it is not unusual anymore to get
large samples of functional observations. For example \cite{CardCC10}  analyze a sample of more than 18000 electricity consumption curves measured every half hour over a period of two weeks. Our study is motivated by the estimation of the central point of a sample of $n=5423$ vectors of $\R^d,$ with $d=86400,$ which correspond to individual television audiences measured every second over a period of 24 hours.

In such large samples of high dimensional data, outlying curves may not be
uncommon and a even  few individuals may corrupt simple statistical indicators such
as the mean trajectory or the principal components (\cite{Ger08}).  Detecting
these atypical curves automatically is not straightforward in such a high
dimensional and large sample context and considering directly robust techniques
is an  interesting alternative.  There are many robust location indicators in
the multivariate setting (\cite{Sma90}) but most of them require high
computational efforts to be estimated, even for small sample sizes, when the dimension is relatively large. For example,  \cite{FrMu01} have extended the notion of trimmed means to a functional context in order to get
robust estimators of the mean profile. In order to deal with the dimensionality issue and to reduce the computation time, \cite{CuevFF07} have proposed random projection techniques  in the context of maximal depth estimators and studied their properties via simulation studies. Note that
sub-sampling approaches based on survey sampling with unequal probability sampling designs have also been proposed
in the literature in order to  reduce the computational time  (\cite{ChGo10}).

\medskip

We focus here on the  geometric median, also called $\lUn$-median or spatial median, 
which is a
direct generalization  of the real median  proposed by \cite{Hal48} and whose properties have been studied in details by \cite{Kem87}.  It can be defined even if the random variable does not have a finite first order moment and it has nice robustness properties since its breakdown point is  equal to 0.5. As noted in \cite{Sma90}, one drawback of the  geometric median is that it is not  affine equivariant. Nevertheless, it is invariant to translation and scale changes and thus is well adapted to functional data which are observed  with the same units at each instant of time.  In a functional context, consistent estimators of the $\lUn$-median have been proposed by \cite{Kem87}, \cite{Cad01} and \cite{Ger08}. Iterative estimation algorithms  have been developed by \cite{Gow74}, \cite{VZ00} in the multivariate setting and by  \cite{Ger08} for functional data. This latter   algorithm requires to invert at each step matrices whose dimension is equal to the dimension $d$ of the data and thus requires important computational efforts. The algorithm proposed by   \cite{VZ00} is much faster and only requires $O(nd)$ operations at each iteration, where $n$ is the sample size. Nevertheless,  these estimation procedures are not adapted   when the  data arrive sequentially, they need to store all the data and   they cannot be simply updated.

\medskip

In this paper, we explore another direction. 
The geometric median being defined as the minimizer of a simple
functional that is differentiable  everywhere when the distribution has
no atom, it is possible to estimate it with  online gradient
algorithms. Such algorithms are very fast and can be simply updated when the data arrive sequentially.
There is a vast literature on stochastic gradient algorithms which
mainly focus on finite dimensional situations (see \cite{Kushner}, \cite{Rup85}, \cite{Benveniste-book90}, \cite{Ljung}, \cite{Duf97},
\cite{KY03}, \cite{Bot10} in the multivariate case, and \cite{ADPY10} on manifolds). The literature is much less abundant when one has
to consider online observations taking values in a functional space (usually an
infinite dimensional Banach or Hilbert space) and most  works focus on 
  linear algorithms (\cite{Wal77}, \cite{DW06},  \cite{SY06}). 
 
\medskip

It is also known in the multivariate setting that averaging procedures can lead to
efficient estimation procedure under additional assumptions on the noise and
when the target is defined as the minimizer of a strictly convex function 
(\cite{PolyakJud92}, \cite{Pel00}). There is little work on averaging when
considering random variables taking values in Hilbert spaces and, as far as we
know, they only deal with linear algorithms (\cite{DW06}). Nevertheless,
it has been  noted in an empirical study whose aim was to estimate the
geometric median with functional data (\cite{CardCC10}) that averaging could
improve in an important way the accuracy of the estimators. 

\medskip

The paper is organized as follows. We first fix notations, give some properties
of the geometric median and present our stochastic gradient algorithm as well
as its averaged version. We also note that our study extends directly to the
estimation of geometric quantiles defined by \cite{Cha96}.  In  a third
section we state the  almost sure consistency and  the $\lDeux$ rates of convergence
of the stochastic gradient estimators as well as the asymptotic normality of
its averaged version. We get that the asymptotic distribution of the averaged
version of the algorithm is the same as the classic estimators. 
A fourth section is devoted to a small simulation study which aims at  comparing the performances of our estimator with the static algorithm developed by   \cite{VZ00}. The comparison is performed according to two points of view, for the same sample size and  for the same computation time. We also analyze  a real example with a large sample of individual television audiences measured every second over a period of 24 hours. The proofs are gathered in Section~\ref{sec=proofs}.

\section{The algorithms and some properties of the geometric median}
\subsection{Definitions and assumptions}\label{sec:defassumpt}

Let $H$ be a separable Hilbert space such as $\xR^d$ or $L^2(I)$, for some closed interval $I \subset \mathbb{R}$. We  denote by $\langle .,.\rangle$  its inner product and by  $\nrm{\cdot}$ the associated norm. 

The geometric median $m$ of a random variable $X$ taking values in $H$  is defined by (see \cite{Kem87}):
\begin{equation}
m \ \egaldef \ \arg \min_{u \in H} \EE{\nrm{X - u} - \nrm{X}} .
\label{defmed}
\end{equation}
Note that this general definition (\ref{defmed}) does not assume the existence of the first order moment of $\nrm{X}.$
We suppose from now on that the following assumptions are fulfilled.
\begin{itemize}
  \item [\bf A1.]  The random variable $X$ is not concentrated on a straight line: 
  for all $v\in H,$  there is $w \in H$ such that $\scal{v,w} = 0$ and
  \begin{equation*}
    \Var ( \scal{w,X} ) > 0.  
  \end{equation*}
  \item [\bf A2.] The law of $X$ is a mixing of two ``nice'' distributions~: $\mu_X = \lambda\muDiff + (1-\lambda)\muDisc$, where
    \begin{itemize}
    	\item $\muDiff$ is not strongly concentrated around single points: if $\ball$ is the ball $\{\alpha \in H, \nrm{\alpha} \leq A\}$, 
	  and $Y$ is a random variable with law $\muDiff$,  
       \begin{equation*}
    \forall A, \exists C_A \in [0,\infty), \forall \alpha \in \ball, \quad
  \EE{\nrm{Y-\alpha}^{-1}} \leq C_A.
  \end{equation*}
\item $\muDisc$ is a discrete measure, $\muDisc = \sum_i p_i \delta_{\alpha_i}$. 
  We denote by $D$ the support of $\muDisc$ and assume that $m\notin D$.
    \end{itemize}
\end{itemize}

As shown in \cite{Kem87}, assumption (A1) ensures that the median $m$ is uniquely defined. 
  The second assumption could probably be relaxed, but it is general enough for most natural examples. 
As noted in \cite{Chaud92},  the conditions on $\muDiff$ are satisfied  when $H=\R^d,$ with $d\geq 2$,  whenever $\muDiff$ has a bounded density on every compact subset of $\R^d$. 
 More precisely, this property is closely related to small ball probabilities since
\[
\EE{ \nrm{Y-\alpha}^{-1}} = \int_0^\infty \prb{\nrm{Y-\alpha} \leq t^{-1}} dt
.\]
 If $\prb{\nrm{Y-\alpha} \leq \epsilon} \leq C \epsilon^d,$ for  some small $\epsilon$ and some positive constant $C,$ it is easy to check that
\[
\EE{\|Y-\alpha\|^{-\beta}} < \infty,
\]
whenever  $0\leq \beta<d.$

When $H = L^2(I),$ the dimension is not finite and  small ball probabilities have been derived  for some particular classes of Gaussian processes (see  \cite{Nazarov2009} for a recent reference). In this case, by symmetry of the distribution, the median $m$ is equal to the mean, and many processes satisfy, for positive constants $C_1, C_2, C_3, C_4$ which depend on the process under study, 
\begin{equation}
\label{def:smallball}
\prb{\nrm{Y-m} \leq \epsilon}   \leq  C_1 \epsilon^{C_4} \exp(-C_2 \epsilon^{-C_3}),
\end{equation}
so that $\EE{ \nrm{Y-m}^{-\beta}} < \infty,$ for all  positive $\beta$. Similar properties of shifted small balls, for $\alpha$ close to $m,$ can be found in \cite{LiShao2001}.

\subsection{Some convexity and robustness properties of the median} 
In this section we derive quantitative convexity bounds
  which will be useful in the proofs. As a consequence, we are also able to bound for the gross sensitivity error, which is a classical robustness indicator   (see  \cite{HubR2009}).

Recalling the definition of the median (eq. \eqref{defmed}), let us  denote by $G : H \mapsto \R$ the function we would like to minimize:
\begin{equation}
  G(\alpha) \egaldef \EE{\nrm {X - \alpha } - \nrm{X}} .
  \label{eq=defG}
\end{equation}
  This function is convex since it is a convex combination  of convex functions. However, convexity is not sufficient to get the convergence of the algorithm.   
  Under assumptions (A1) and (A2)
this  function can be decomposed in two parts:
\[
G(\alpha) = \lambda G_c(\alpha) + (1 - \lambda) G_d(\alpha), 
\]
where the discrete part $G_d(\alpha) = \sum_i p_i (\nrm{x_i - \alpha} - \nrm{x_i})$ has been isolated. 
The first part is Fr\'echet differentiable everywhere (\cite{Kem87}), so $G$ is differentiable except on $D$, the support of the discrete part $\muDisc$ . 
We denote by $\Phi = \lambda\Phi_c + (1 - \lambda) \Phi_d$ its Fr\'echet derivative,
\begin{equation*}
\Phi(\alpha) \egaldef   \nabla_\alpha G
= - \EE{\frac{ X - \alpha}{ \nrm{ X- \alpha}}}.
\end{equation*}
  \begin{rem}
    It will be useful to define $\Phi$ on the set $D$. If $x\in D$, we define $G_x$ by ``forgetting'' $x$, 
    \[ G_x(y) = \sum_{i, x_i \neq x} p_i (\nrm{x_i - y} - \nrm{x_i}).\]
    This function is Fréchet differentiable in $x$, and we let
    \[ \Phi_d(x) = \sum_{i, x_i \neq x} p_i\frac{x - x_i}{\nrm{x - x_i}}.\]

    In the vocabulary of convex analysis, we have just chosen a particular subgradient of $G$
    on the set $D$ of non-differentiability. It is easily seen that:
    \begin{equation}
      \label{eq:subgradient}
    \forall x,y, \quad G(y) - G(x) \geq \scal{\Phi(x), y-x},
  \end{equation}
which asserts that $\Phi$ is a subgradient.  A short proof of inequality \eqref{eq:subgradient} is given in Section~\ref{sec:proof1}.
  \end{rem}
The median $m$  is then the unique solution of the nonlinear equation,
\begin{equation}
  \Phi(\alpha) = 0.
  \label{eqmoment}
\end{equation}

To exhibit some useful strong convexity and robustness properties of the median we need to introduce the Hessian of functional $G$,
for $\alpha \in H\setminus D$.  It is denoted by $\Gamma_\alpha,$ maps $H$ to $H$ and  it is easy to check (see \cite{Kol97} for the multivariate case and \cite{Ger08} for the functional one) that
\begin{equation*}
\Gamma_\alpha 
 =  \E \left[ \frac{1}{\| X-\alpha\|} \left(
    \idt - \frac{( X-\alpha) \otimes  (X-\alpha)}{\| X-\alpha\|^2} \right)
  \right],
\end{equation*}
where $\idt$ is the identity operator in $H$ and
  $u \otimes v (h) = \scal{u,h}v,$ for $u,v$ and $h$ belonging to $H.$
The operator $\Gamma_\alpha$ is not compact but it is bounded when $\EE{\nrm{X-\alpha}^{-1}}< \infty.$

If we define $\bar{h} = h/\nrm{h}$, and $P_h$ the projection onto the orthogonal complement of $h$, 
\begin{align}
  \scal{h,\Gamma_\alpha h}
  &= \nrm{h}^2 \E\left[
  \frac{1}{\nrm{\alpha - X}} \left( 1 - \frac{\scal{\bar{h},\alpha - X}^2}{\nrm{\alpha - X}^2} \right)
    \right] \notag\\
    \label{eq=niceFormula}
  &= \nrm{h}^2 \E\left[
  \frac{1}{\nrm{\alpha - X}} \frac{\nrm{P_{\bar{h}}(\alpha - X)}^2}{\nrm{\alpha - X}^2}
    \right].
\end{align}

We can now state  a strong convexity property of functional $G$ which can be seen as an extension to an infinite dimensional setting of Proposition 4.1 in  \cite{Kol97}. 
 \begin{prop}
   \label{prp=GestConvexe}
   Recall that $\ball$ is the ball of radius $A$ in $H$.  Under assumptions A1 and A2, there is  a strictly positive constant $c_A$, such that:
   \[
   \forall \alpha \in \ball \setminus D, \forall h \in H, \quad
   c_A \nrm{h}^2  \leq \scal{h,\Gamma_\alpha h} \leq C_A \nrm{h}^2. 
   \]
 \end{prop}
   In other words, $G$ is strictly convex in $H$ and it is strongly convex on any bounded set, as shown in the following corollary. 
 \begin{cor}  \label{cor=GestConvexe}
 Assume hypotheses of Proposition \ref{prp=GestConvexe} are fulfilled.
   For any strictly positive  $A$, there is a strictly positive constant $c_A$ such that:
   \begin{align*}
     &\forall \alpha_1,\alpha_2 \in \ball^2, \quad &
     \scal{ \Phi(\alpha_2) - \Phi(\alpha_1), \alpha_2 - \alpha_1 } 
     &\geq c_A \nrm{\alpha_2 - \alpha_1}^2. 
   \end{align*}
 \end{cor}
 
 As a  particular case of Proposition \ref{prp=GestConvexe}, we get that there exist two strictly positive constants $0<c_m \leq C_m \leq \EE{ \nrm{X-m}^{-1}} < \infty,$ such that
\begin{equation}
c_m \nrm{h}^2  \leq \scal{h,\Gamma_m h} \leq C_m \nrm{h}^2. 
\label{eq=bndsGammam}
\end{equation}

As noted in \cite{Kem87}, the geometric median has a 50 \% breakdown point. 
Furthermore, an immediate consequence of (\ref{eq=bndsGammam}) is that operator $\Gamma_m$ has a bounded inverse. Thus, the gross error sensitivity, which  is also a classical indicator of robustness, is bounded for the median in a separable Hilbert space. Indeed, thanks to the expression derived in \cite{Ger08}, it is bounded as follows,
\[
\sup_{z \in H} \  \nrm{ \Gamma_m^{-1} \left( \frac{z-m}{\nrm{z-m}} \right) } \leq \frac{1}{c_m}. 
\]

\subsection{The algorithms}

Given $X_1, X_2, \ldots, X_n$, $n$ independent copies of $X$, a natural estimator of $m$ is  the solution $\widehat m_n$ of the empirical version of  \eqref{eqmoment},
\[
\sum_{i=1}^n \frac{X_i - \widehat m_n}{\left\| X_i - \widehat  m_n \right\|}  = 0.
\]
The solution $\widehat{m}_n$ is defined implicitly and is found by iterative algorithms.

We propose now  an alternative and simple  estimation algorithm which can be  seen as a stochastic gradient algorithm (\cite{Rup85,Duf97}) and is defined as follows
\begin{align}
  Z_{n+1} &= Z_n + \gamma_n \frac{ X_{n+1} - Z_n}{\nrm{ X_{n+1} - Z_n }} \notag\\
\label{def=algobase}
  &= Z_n - \gamma_n \stp{n+1},
\end{align}
with a starting point that can be random and bounded, \textit{e.g.} $Z_0 = X_0 \ind{\{ \nrm{X_0}\leq M \}}$ for some positive  constant $M$ fixed in advance, or deterministic.
If $X_{n+1} = Z_n$, we set $\stp{n+1} = 0$ and $Z_{n+1} = Z_n$ so the algorithm does not move.
The sequence of descent steps $\gamma_n$ controls the convergence of the algorithm. The direction $\stp{n+1}$ is an ``estimate'' of the gradient $\Phi$ of $G$  at $Z_n$ since the conditional expectation given  the sequence of $\sigma$-algebra $\CF_n=\sigma(Z_1,\ldots, Z_n)=\sigma(X_1,\ldots, X_n)$ satisfies
\begin{equation}
  \EFn{\stp{n+1}} = \Phi(Z_n).
  \label{eq=step}
\end{equation}
Note that our particular choice of  subgradient $\Phi$ ensures that this equality always holds. 

Defining now by  $\xi$ the sequence of ``errors'' in these estimates,
\begin{equation}
  \xi_{n+1} = \Phi(Z_n) - \stp{n+1},
  \label{eq=defXi}
\end{equation}
algorithm (\ref{def=algobase}) can also be seen as a non linear Robbins-Monro algorithm,
\begin{align}
  \label{eq=algoRM}
  Z_{n+1} &= Z_n + \gamma_n \left( -\Phi(Z_n)  + \xi_{n+1} \right).
\end{align}
Thanks to \eqref{eq=step} and \eqref{eq=defXi}, the sequence $(\xi_n)$ is a sequence of  martingale differences. Let us note that  the bracket of the associated martingale satisfies,
\begin{align}
  \EFn{\nrm{\xi_{n+1}}^2}
  \nonumber
  &= \EFn{\nrm{\stp{n+1}}^2} + \nrm{\Phi(Z_n)}^2 - 2 \scal{\Phi(Z_n), \EFn{\stp{n+1}}} \\
  \nonumber
  &= \EFn{\nrm{\stp{n+1}}^2}  - \nrm{\Phi(Z_n)}^2 \\
  &\leq 1 - \nrm{\Phi(Z_n)}^2 \leq 1. 
  \label{eq=crochetXi}
\end{align}

Our second algorithm consists in averaging all the estimated past values,
\[
\Zbar_{n+1} =\Zbar_{n} + \frac{1}{n+1} \left( Z_{n+1} -\Zbar_n\right) 
\]
with $\Zbar_0=0,$ so that $\Zbar_n = \frac{1}{n} \sum_{i=1}^n Z_i.$

\begin{rem}
An extension of the notion of  quantiles in Euclidean and Hilbert spaces has been proposed by  \cite{Cha96}.
In such spaces, quantiles are associated to a direction and a magnitude specified by a vector $v \in H,$ such that $\nrm{v}<1$. 
 The geometric quantile of $X,$ say $m^v,$ corresponding to direction $v$ and magnitude $\nrm{v}$  is defined, uniquely under previous assumptions, by
 \[
m^v = \arg\min_{u \in H} \EE{ \nrm{X - u} + \scal{X-u,v}}.
\] 
If $v=0$ one recovers the geometric median. When $\nrm{v}$ is close to one, $m^v$ is a (directed) extreme quantile. 
In any case, $m^v$ is characterized by:
\[
\Phi_v(m^v) = \Phi(m^v) - v = 0,
\]
so that it can be naturally estimated with  the  following stochastic algorithm 
\[
\widehat{m}_{n+1}^v = \widehat{m}_n^v + \gamma_n \left( \frac{ X_{n+1} - \widehat{m}_n^v}{\left\| X_{n+1} - \widehat{m}_n^v \right\| } + v \right) ,
\]
as well as with its averaged version. 
\end{rem}

\section{Convergence results}\label{sec=median}

\subsection{Almost sure convergence of the stochastic gradient algorithm}
\label{sec=cvgMedian}

We first state the almost sure consistency of our sequence of estimators $Z_n$ under classical and general assumptions on the descent steps $\gamma_n.$
\begin{thm}
  If (A1) and (A2) hold, and if $(\gamma_n)_{n\in\xN}$ satisfies the usual conditions:
  \[
  \sum_{n \geq 1} \gamma_n = \infty, 
  \qquad
  \sum_{n \geq 1} \gamma_n^2 <  \infty,
  \]
then 
\[
\lim_{n \rightarrow \infty} \| Z_n - m \| = 0, \quad a.s.
\]
\label{prop:convH}
\end{thm}



\subsection{Rates of convergence and asymptotic normality}
We present now the rates of convergence of the stochastic gradient algorithm as well as the asymptotic distribution of its averaged version. The proofs  are given in Section~\ref{sec=proofs}. 
More specific sequences $(\gamma_n)$ are considered and we suppose from now on that  $\gamma_n=c_\gamma n^{-\alpha}$, where $c_\gamma$ is a positive constant and $\alpha \in(\frac{1}{2},1)$. 
We need one additional assumption to get these rates of convergence: 

\begin{itemize}
	\item [\bf A3.] There is a positive constant $A$ such that
 \begin{equation}
 \label{moment2}
 \exists C_A \in [0,\infty), \forall h \in \ball, \quad
  \EE{\nrm{X-(m+h)}^{-2}} \leq C_A.
 \end{equation}
\end{itemize}
This assumption is not restrictive when the dimension is strictly larger than two as discussed in Section \ref{sec:defassumpt}. 

The following proposition states that, on events 
of arbitrarily high probability, the functional estimator $Z_n$ attains the  classical rates of convergence in quadratic mean (see \cite[theorem 2.2.12]{Duf97} for the multivariate case) up to a logarithmic factor.  
\begin{prop}
  \label{lmm=vitesseQuadratique}
  Assume (A1), (A2)  and (A3).  
 Then,  there exist an increasing  sequence of events $(\Omega_N)_{N\in\xN}$, and constants $C_N$, such that $\Omega = \bigcup_{N\in\xN} \Omega_N$, and
  \begin{equation}
    \forall N, \quad 
    \EE{\ind{\Omega_N} \nrm{Z_n -m}^2}
    \leq C_N
      \gamma_n\ln\left(\sum_{k=1}^n\gamma_k \right)
    \leq C_N \frac{\ln(n)}{n^\alpha} . \nonumber
  \end{equation}
\end{prop}
\begin{rem}
  An immediate consequence of Proposition \ref{lmm=vitesseQuadratique} is that 
  \[
  \nrm{Z_n - m }^2 = O_P \left( \frac{\ln(n)}{n^\alpha}\right). 
  \]
\end{rem}

 Assumption (A3) is needed to bound the difference between $G$ and its quadratic approximation, in a neighborhood of $m$ as stated in the following Lemma.

\begin{lem}
Assume (A3) is in force. Then,
\begin{equation*}
\forall h \in \ball, \quad \Phi(m+h)  =  \Gamma_m(h) + O\left(\nrm{h}^2\right).
\end{equation*}
\label{def:TaylorPhi}
\end{lem}
Finally, Theorem~\ref{prop:vitesse} stated below probably gives the most important result of this work. It is shown that the averaged estimator $\Zbar_n$  and the classic static estimator $\widehat{m}_n$ have the same asymptotic distribution.  Consequently, for large sample sizes, it is possible to get, very quickly, estimators which are as efficient, at first order, as the slower static one $\widehat{m}_n.$
Note that the asymptotic distribution of $\widehat{m}_n$ has  been derived in the multivariate case by \cite{hab89}, Theorem 6.1. For variables taking values in a Hilbert space, such asymptotic distribution has only been proved  for  a particular case, when the support of $X$ is  a finite dimensional space  (Theorem 6 in \cite{Ger08}). 

\begin{thm}
Assume (A1), (A2) and (A3). 
Then, 
\[
  \sqrt{n}\left( \Zbar_n - m \right)
  \cvl
  \mathcal{N} \left(0, \Gamma_m^{-1}\covLimite \Gamma_m^{-1}\right),
\]
with, 
\[
\covLimite =  \EE{\frac{(X-m)}{\nrm{X-m}} \otimes \frac{(X-m)}{\nrm{X-m}}}.
\]
\label{prop:vitesse}
\end{thm}
Note that with (\ref{eq=bndsGammam}), operator $\Gamma_m^{-1}$ is well defined, it is  bounded and positive.



\section{Illustrations on simulated and real data}

\subsection{A simulation study}
A simple simulation study is performed to check the good behavior of the averaged stochastic estimator and to make a comparison with the static estimator developed by  \cite{VZ00}.   Two points of view are considered. The first classic one consists in evaluating the performances of these two different approaches for different sample sizes.  The second one, which is the point of view that should be adopted when computation time matters, consists in comparing the accuracy of both approaches when the allocated computation time is fixed in advance. 
We use \Rlogo{} (\cite{R10}) and the function \texttt{spatial.median} from the library \texttt{ICSNP} to estimate the median with  the algorithm developed by \cite{VZ00}.

For simplicity, we consider random variables taking values in $\R^3$ and make simulations of Gaussian random vectors with median $m=(0,0,0)$ and covariance matrix:
\[
\Gamma = \begin{pmatrix} 3 & 2 & 1 \\ 2 & 4 & -0.5 \\ 1&  - 0.5 & 2 \end{pmatrix}.
\]
In order to compare the accuracy of the different algorithms, we   compute the following estimation error,  
\begin{equation}
\label{sim:crit}
R(\widehat{m}) = \nrm{ \widehat{m} - m},
\end{equation}
where  $\widehat{m}$ is an estimator of $m.$

Our averaged estimator depends on the tuning parameters $\alpha$ and $c_\gamma$ which control the descent steps $\gamma_k = c_\gamma k^{-\alpha}.$ It is well known that for the particular case $\alpha=1$, the choice of parameter $c_\gamma$ is crucial for the convergence and depends on the second derivative of $G$ in $m$ which is unknown in practice. As usually done for such procedures, we fix $\alpha=3/4$ and focus on the choice of $c_\gamma.$ 
We also run in parallel the algorithm for 10 initial points chosen randomly in the sample and then select the best estimate  $\widehat{m}$ which corresponds to  the minimum value of
\[\alpha \mapsto \frac1n\sum_{i=1}^n\left(\nrm{X_i - \alpha} - \nrm{X_i}\right),\]
which is the empirical version of (\ref{eq=defG}).

\subsubsection{Fixed sample sizes}
We perform 1000 simulations for different sample sizes, $n=250,$  $n=500$ and $n=2000.$ 
Table~\ref{table1} presents basic statistics for the estimation errors (first quartile $Q_1$, median and third quartile $Q_3$), according to criterion \eqref{sim:crit},  for the algorithm by \cite{VZ00} and our averaged procedure considering different values for $c_\gamma \in \{ 0.2, 0.6, 1, 2, 5, 10, 15, 25, 50, 75 \}.$

\begin{table}[htdp]
\caption{Comparison of the estimation errors for different sample sizes}
\begin{center}
\begin{tabular}{|c|ccc|ccc|ccc|} \hline
 &  & n=250 & &  & n=500 &  &  & n=2000 & \\ \hline
Estimator & [Q1 & median & Q3] & [Q1 & median & Q3] & [Q1 & median & Q3] \\ \hline
 $c_\gamma = 0.2$ & 0.45 & 0.60 &  0.80 & 0.38& 0.53 & 0.69  & 0.25 & 0.35 &  0.47 \\
 $c_\gamma = 0.6$ & 0.21 & 0.29 &  0.40 &  0.15 &0.21 &  0.29 & 0.06 &0.09 &  0.12 \\
 $c_\gamma = 1$ &0.15 & 0.22 &  0.31 &  0.11 & 0.16 &  0.21 & 0.05 & 0.08 &  0.10 \\
 $c_\gamma = 2$  & 0.15 & 0.21 & 0.30 & 0.09 & 0.15 &  0.20  & 0.05 & 0.07 &  0.10 \\
$c_\gamma = 5$ &  0.13 &  0.19 &  0.25 & 0.09 & 0.13 &  0.18 & 0.04 & 0.06 & 0.09 \\
$c_\gamma = 10$ & 0.13 & 0.18 & 0.25  & 0.09 & 0.13 & 0.18  &0.04 & 0.06& 0.09 \\
$c_\gamma = 15$  &0.12 & 0.18 &  0.25 &  0.09 & 0.13 &  0.18 &0.04 & 0.06& 0.08 \\ 
$c_\gamma = 25$  &0.13 & 0.19 &  0.26 &  0.09 & 0.13 &  0.18 &0.04 & 0.06& 0.09 \\ 
$c_\gamma = 50$  &0.13 & 0.19 &  0.26 &  0.09 & 0.13 &  0.18 &0.04 & 0.06& 0.09 \\ 
$c_\gamma = 75$  &0.14 & 0.20 &  0.27 &  0.09 & 0.14 &  0.19 &0.05 & 0.07& 0.09 \\ \hline
 Vardi \& Zhang & 0.12 & 0.18 & 0.25 & 0.09 & 0.12 & 0.17  & 0.04 & 0.06 &  0.08 \\ \hline
\end{tabular}
\end{center}
\label{table1}
\end{table}

At first, we can note that even for moderate sample sizes the averaged procedure performs well in comparison with the Vardi and Zhang estimator which only gives slightly better estimations. 
We can also remark that the averaged stochastic estimator is not much sensitive to the tuning parameter $c_\gamma$ which can take values in the interval $[2,75]$ without modifying the performances of the estimator. As a matter of fact, we noted on simulations that interesting values for $c_\gamma$ are around or above $\EE{\nrm{X-m}},$ which is about 2.7 for this particular simulation study. 

\subsubsection{Fixed computation time}

Even if both algorithms require  computation times which are $O(nd)$ (for $n$ observations in dimension $d$),  the averaged stochastic gradient approach is much faster (on the same computer, with procedures coded in the same \Rlogo \  language).
For example, in previous simulations, if the sample size is $n=1000,$ the
averaged estimator is about 30 times faster. When the dimension gets larger the difference is even more impressive, 
as we will see in the next section. 

 Let us suppose the allocated time for computation is limited and fixed in advance, say 1 second,  and compare the  sample sizes  that can be handled by the different algorithms. The static estimator by \cite{VZ00} can deal with $n=150$ observations, whereas our recursive algorithm, coded in the \Rlogo \  language, can take into account $n=4500$ observations so that it will gives much better estimates of the median, as seen in Table  \ref{table1}. Finally, if the algorithm is coded in \Clang{} and called from \Rlogo, then it is at least 20 times faster than its \Rlogo \  analogue, so that it can deal with at least $n=90000$ observations, during the same second.

\subsection{Estimation of the  median television audience profile}\label{sec:Mediametrie}
The analysis of  audience profiles for different  channels, or different days of the year, is an essential tool to understand the consumers' habits as regards television. 
The French company \MM{} provides official television audience rates in France. \MM{} works with a panel of about 9000 individuals and the television sets of these individuals are equipped with sensors that  measure the audience of the different channels at a second scale.

A sample of around 7000 people is drawn every day in this panel and the television consumption of the people belonging to this sample is recorded every second. The data are then sent sequentially to \MM{} during the night.  Survey sampling techniques with unequal probability sampling designs are used by \MM{} to select the sample and thus  the i.i.d assumption is clearly not satisfied. Nevertheless, our aim is just to give an illustration of the ability of our averaged stochastic algorithm to deal with a large sample of very high dimensional data.  Moreover, \MM{}  has noted  in these samples the presence of  some atypical behaviors so that robust techniques may be helpful. 

We focus our study on the estimation of  the television audience profile during the 6th september 2010. After removing from the  sample people that did not watch television at all on that day, we finally get a sample of size $n=5423.$ For each element $i$ of the sample, we have a vector $X_i \in \{0,1\}^{86400},$ where $86400$ is the number of seconds within a day,  and zero values correspond to seconds during the day where $i$ is not watching television. 

A classical audience indicator is given by the mean profile, drawn in Figure \ref{fig1},  which is simply the proportion of people watching television at every second over the considered period of time. We compare this classical indicator with the geometric median, whose estimation is drawn in black in Figure \ref{fig1}.
We can first note that both estimators have the same shape along time, showing three peaks of audience during the day with higher audience rates  between 8 and 10 PM.  Estimated values are smaller for the geometric median which is less sensitive to small perturbations and outliers.  This also indicates that the distribution of the individual audience curves is not symmetric around the mean profile.

From a computational point of view, even if the database is huge, it takes less than one minute for our algorithm to converge whereas we were not able to perform the estimation with the static estimator developed by  \cite{VZ00} because of memory contraints. The value of the tuning parameter was chosen to be $c_\gamma=400,$ it leads to a value of about 92 for the empirical loss criterion.
   

  \begin{figure}[ht]
   \begin{center}
  \includegraphics[height=14cm,width=15.5cm]{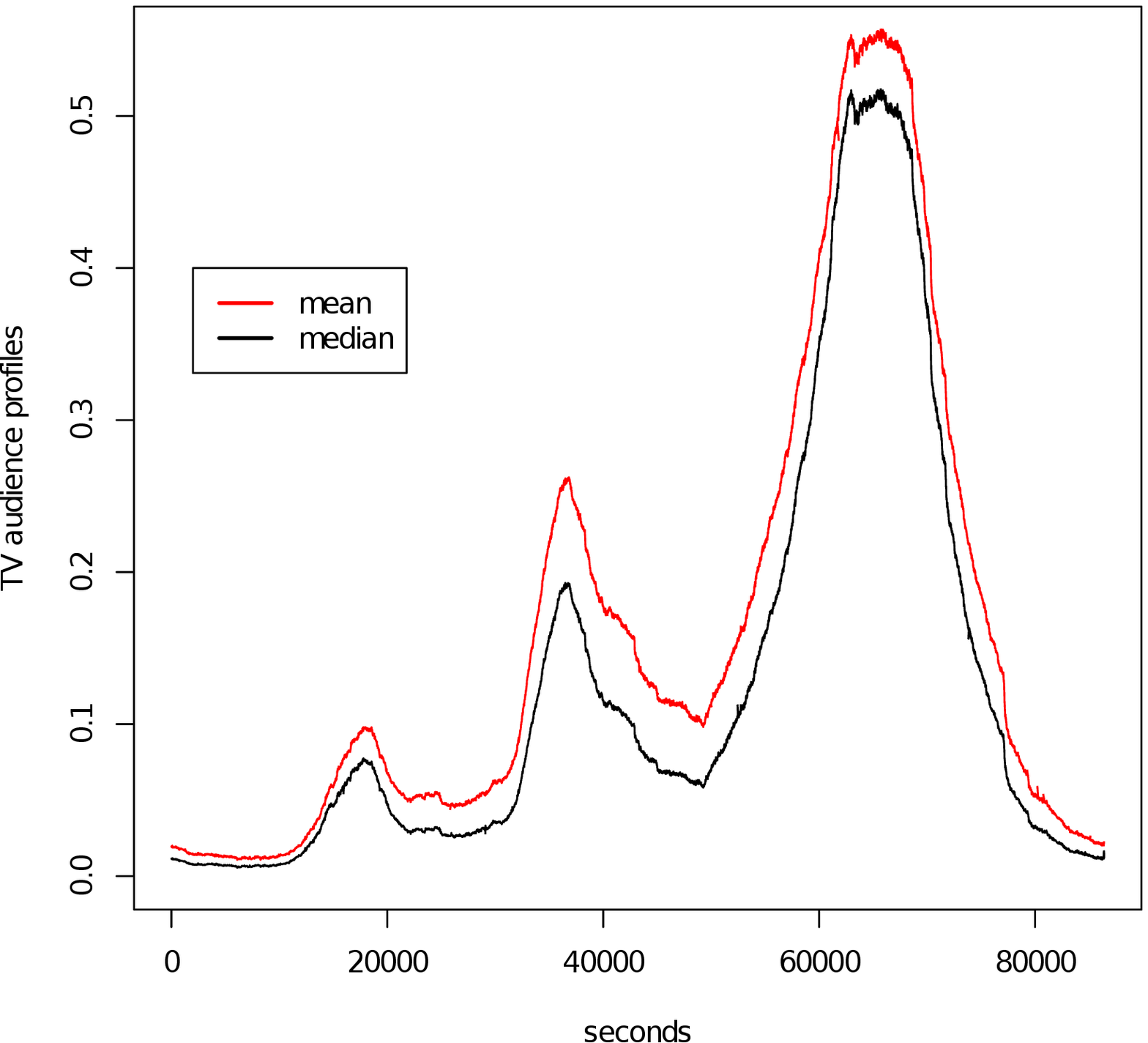}
 \caption{Estimations of the mean and the geometric median audiences, at a second scale, during the 6th september 2010.}
 \label{fig1}
 \end{center}
   \end{figure}

\section{Concluding remarks}
The experimental results confirm that averaged recursive estimators of the geometric median relying on stochastic gradient approaches are of particular interest when one has  to deal with large samples of data and potential outliers. Furthermore, when the allocated computation time is limited and fixed in advance and the data arrive online  these techniques can deal, in a recursive way, with larger sample sizes and finally provide estimations that are much more accurate than static estimation procedures.
We have also noted in the simulation experiment that they are not very sensitive to the value of the tuning parameter~$c_\gamma.$ 

One could imagine many directions for future research that certainly deserve further attention. 
Taking advantage of the rapidity of our estimation procedure, one could use resampling techniques, similar to the bootstrap, in order to approximate the asymptotic distribution of the estimator given in Theorem~\ref{prop:vitesse} and  then  build pointwise confidence intervals. Proving rigorously the  validity of such techniques is far beyond the scope of this paper. 

Our procedure can also be extended readily for online clustering, adapting the well known MacQueen algorithm (\cite{MacQueen}) to the $L_1$ context. Even if the criterion to be optimized is not convex anymore, it can be proved that stochastic gradient approaches converge almost surely to the set of stationary points (\cite{CCM10}) and thus are interesting candidates for online clustering.

Another direction of interest is online estimation of the conditional geometric median when real covariates are available.  For instance, the age or the size of the city where individual live are known by M\'ediam\'etrie and it can be possible to take such information into account in order to get varying time regression models that can also be estimated in a very fast way thanks to sequential approaches.

Finally, as noted in Section~\ref{sec:Mediametrie} the independence condition is rarely satisfied for real studies. A direction that deserves further investigation is to determine under which dependence conditions our results, such as Theorems \ref{prop:convH} and \ref{prop:vitesse},  remain true.

\section{Proofs}
\label{sec=proofs}
\subsection{Convexity --- Proofs of Proposition \ref{prp=GestConvexe} and Corollary  \ref{cor=GestConvexe} }\label{sec:proof1}
We first show that $\Phi$ is a subgradient of $G$. For points $x \notin D$, this is clear since $G$ is Fréchet differentiable. 

Pick a point $x_0$ in $D$ and recall that $\Phi_d(x_0) = \sum_{i, x_i \neq x_0} p_i \frac{x_0 - x_i}{\nrm{x_0 - x_i}}$. We have, 
\begin{align*}
  \scal{\Phi_d(x_0), y - x_0} 
  &= \sum_{i, x_i\neq x_0} p_i \frac{ \scal{x_0 - x_i, y - x_0}}{\nrm{x_0 - x_i}} \\
  &= \sum_{i, x_i\neq x_0} p_i \frac{ \scal{x_0 - x_i, y - x_i}}{\nrm{x_0 - x_i}} 
    -\sum_{i, x_i\neq x_0} p_i \nrm{x_0 - x_i} \\
  &\leq \sum_{i, x_i\neq x_0} p_i \nrm{y- x_i} 
    -\sum_{i, x_i\neq x_0} p_i \nrm{x_0 - x_i} \\
  &\leq G_d(y) - G_d(x_0),
\end{align*}
so that $\Phi_d$ is a subgradient of $G_d.$

The upper bound in Proposition \ref{prp=GestConvexe} follows immediately from \eqref{eq=niceFormula} and the assumption (A2).

   For the lower bound, thanks to \eqref{eq=niceFormula}, we only need to prove:
   \begin{align}
     \label{eq:cvxBndI}
     &\forall \alpha\in\ball,\forall u, \nrm{u} = 1,\quad & 
     \scal{u,\Gamma_\alpha u} = 
      \E\left[ \frac{\nrm {P_u(X-\alpha)}^2}{\nrm{X - \alpha}^3} \right]
      &\geq c_A,
   \end{align}
   where $P_u$ is the projection on the orthogonal of $u$. This quantity is small when $X-\alpha$ is in $\vect{u}$. 
   
   Recall that (by (A1)), $X$ is not supported on a line. 
   Consider the set of subspaces $K\subset H$ satisfying: $\forall x\in K, \Var(\scal{x,X}) = 0$. Suppose
   that this set is non-empty, and let $H'$ be a maximal element in it (this exists by Zorn's lemma). 
   The orthogonal of $H'$ has at least dimension $2$ (otherwise, we get a contradiction to A1). Let $v_1,v_2$ 
   be two orthogonal vectors in $H'^\perp$. Let $v_t = \cos(t)v_1 + \sin(t)v_2$. The map
   \[ t \mapsto \Var(\scal{v_t, X}) \]
   is continuous on a compact set. Its minimum cannot be zero (since this would contradict the maximality of $H'$). 
   Therefore there exists  a $c$ such that, for all unit $v$ in  
   the plane spanned by $(v_1,v_2)$, $\Var(\scal{X,v}) \geq c$.

   The  orthogonal of $u$ (an hyperplane) and the (2-dimensional) plane spanned by $v_1$ and $v_2$ necessarily intersect: there exists a unit vector $v\in\vect{v_1,v_2}$ such that $\scal{u,v} = 0$. Therefore, for all $y\in H$,  $\nrm{P_u y}^2 \geq \scal{y,v}^2$. In particular,  $\nrm{P_u (X-\alpha)}^2 \geq \scal{X - \alpha,v}^2$.

   Suppose first that $X$ is a.s. bounded by $K$. Then
   \[ \E\left[ \frac{\scal{v,X-\alpha}^2}{\nrm{X - \alpha}^3} \right]
   \geq  \frac{1}{(A+K)^3} \E\left[ \scal{v,X-\alpha}^2\right].  \]
   It is easily seen that the last term is bounded below by $\Var( \scal{v,X})\geq c$ and \eqref{eq:cvxBndI} holds with
   \[ c_A = \frac{1}{(K+A)^3}c.\] 
   To get rid of the boundedness assumption on $X$, we can  just choose $K$ large enough so that $\Var(\scal{v,X\ind{\nrm{X}\leq K}})$ is strictly positive for $v = v_1,v_2$.  
   
   \bigskip 

  The corollary is a consequence of Proposition \ref{prp=GestConvexe} and of the fact that $\Phi$ is a subgradient. 
Indeed, the inequality holds for $\Phi_c$ by interpolation: for an elementary proof, define $\alpha_t = (1-t)\alpha_1 + t \alpha_2$, and write
   $\Phi(\alpha_2) - \Phi(\alpha_1) = \int_0^1 f'(t) dt$ where $f(t) = \Phi(\alpha_t)$. One can then apply \eqref{eq:cvxBndI}, $t$ by $t$, with $\alpha = \alpha_t$ and $u = \frac{\alpha_2 - \alpha_1}{\nrm{\alpha_2 - \alpha_1}}$. 

   Moreover, 
   \begin{align*}
      \scal{\Phi_d(\alpha_2) - \Phi_d(\alpha_1),\alpha_2 - \alpha_1}
      &= \scal{\Phi_d(\alpha_2), \alpha_2 - \alpha_1} + \scal{\Phi_d(\alpha_1), \alpha_1 - \alpha_2} \\
      &\geq G(\alpha_2) - G(\alpha_1) + G(\alpha_1 ) - G(\alpha_2) \\
      &= 0,
   \end{align*}
   where the second line follows from \eqref{eq:subgradient}. 
   Since $\Phi = \lambda\Phi_c + (1 - \lambda)\Phi_d$, Corollary  \ref{cor=GestConvexe} is proved.

\subsection{Proof of Theorem \ref{prop:convH}.}

The proof of Theorem \ref{prop:convH}  follows  a classical strategy and consists of two steps.

\begin{lem} Under the hypotheses of Theorem \ref{prop:convH}, there is a random variable $V$ such that, $\EE{|V|^2} < \infty,$ and 
\begin{align*}
\lim_{n \rightarrow \infty} \parallel Z_{n}-m\parallel^2 & = V, \quad a.s.
\end{align*}
\label{lem1:convH}
\end{lem}

\begin{proof}[Proof of Lemma \ref{lem1:convH}]
Let us consider
 $ V_n \egaldef  \nrm{ Z_{n}-m}^2 .$
Recall that $Z_{n+1} = Z_n - \gamma_n \Phi(Z_n) + \gamma_n \xi_{n+1}$ (cf. \eqref{eq=algoRM}). Therefore 
\begin{align*}
  V_{n+1} &= \nrm{Z_n - m - \gamma_n \Phi(Z_n)}^2 + \gamma_n^2 \nrm{\xi_{n+1}}^2
     + 2 \gamma_n \scal{ \xi_{n+1}, Z_n - \gamma_n \Phi(Z_n) }.
\end{align*}
If we condition with respect to $\CF_n$, the last term disappears since $(\xi_{n})$ is a martingale difference sequence and it comes:
\begin{align}
  \notag
  \EFn{ V_{n+1}} 
  &= \nrm{Z_n - m - \gamma_n \Phi(Z_n)}^2 + \gamma_n^2 \EFn{\nrm{\xi_{n+1}}^2} \\
  \notag
  &= \nrm{Z_n - m}^2 - 2 \gamma_n\scal{Z_n - m, \Phi(Z_n)} 
  + \gamma_n^2 \left( \nrm{\Phi(Z_n)}^2 + \EFn{\nrm{\xi_{n+1}}^2} \right) \\
  \label{eq=decompositionVn}
  &= V_n - 2 \gamma_n\scal{Z_n - m, \Phi(Z_n)} + \gamma_n^2,
\end{align}
where we used the definition of $V_n$ and \eqref{eq=crochetXi} for the last term.
Since $G$ is convex,
using Corollary \ref{cor=GestConvexe}, we get:
\[ \scal{Z_n - m, \Phi(Z_n)} = \scal{Z_n - m, \Phi(Z_n) - \Phi(m)} \geq 0. \]
Therefore, for all $n$, 
\(
\E\left[V_{n+1}|\CF_n\right] \leq  V_n +\gamma_n^2.
\)
From the Robbins Siegmund theorem (see for instance \cite[page 18]{Duf97}), we deduce that $(V_n)$ converges almost surely to $V$. Moreover, we note  that
$Z_n - m$ is bounded in $L^2,$
\begin{equation}
  \label{eq=VnBornee}
  \forall n,\quad V_n = \EE{\nrm{Z_n - m}^2} \leq \EE{\nrm{Z_0 - m}^2} +  \sum_{k=1}^\infty \gamma_k^2 < \infty,
\end{equation}
whenever $\EE{\nrm{Z_0 - m}^2} < \infty,$ which is satisfied for example if $Z_0 = X_0 \ind{\{ \nrm{X_0}\leq M \}}, $ with $M< \infty.$
\end{proof}

We can now give the proof the  theorem.
\begin{proof}[Proof of Theorem \ref{prop:convH}]  
Lemma \ref{lem1:convH} shows that the sequence $V_n$ converges almost surely. Let us check now that its limit is zero. Let us take expectations in equation \eqref{eq=decompositionVn}:
\begin{align*}
  \EE{ V_{n+1}} 
& = \EE{ V_n} + \gamma_n^2 - 2\gamma_n \EE{ \scal{\Phi(Z_n), Z_n-m } } \nonumber \\
& = \EE{ V_0}  + \sum_{k=1}^{n} \gamma_k^2 
- 2   \sum_{k=1}^{n} \gamma_k \EE{\scal{ \Phi(Z_k), Z_k-m }}.
\end{align*}
The sequence $\sum_{k=1}^n \gamma_k \EE{\scal{\Phi(Z_k), Z_k - m}}$ has positive terms, and is bounded above by $\EE{V_0} + \sum_{k=1}^\infty \gamma_k^2$, therefore it converges. This implies in particular that 
\begin{align}
  \label{eq=cestFini}
  \sum_{n=1}^{\infty} \gamma_n \scal{\Phi(Z_n), Z_n-m}< + \infty \quad\text{a.s.}
\end{align}
This convergence cannot happen unless $Z_n$ converges to $m$. Indeed, for each $\epsilon \in ]0,1[,$ let us introduce the set
\[
  \Omega_\epsilon = \left\{ 
     \omega \in \Omega : 
     \ \exists n_\epsilon(\omega)\geq 1,  
     \forall n \geq n_\epsilon(\omega), 
  \quad
     \epsilon^2 < V_n(\omega) < \epsilon^{-2}
  \right\}.
\]
For $\omega \in \Omega_\epsilon,$ we have with Corollary  \ref{cor=GestConvexe},
\[
  \sum_{n\geq 1} \gamma_n \scal{ \Phi(Z_n(\omega)), Z_n(\omega)-m}
  \geq \left( \sum_{n\geq n_\epsilon(\omega)} \gamma_n \right)
  \inf_{\epsilon < \nrm{\alpha - m} < \epsilon^{-1} } \scal{\Phi(\alpha), \alpha -m}
  = \infty,
 \]
 which contradicts \eqref{eq=cestFini} unless $\PP(\Omega_\epsilon) = 0$. 
 Since $V_n$ converges a.s. to a finite limit, and 
 $\{\lim V_n \in[c, c^{-1}]\} \subset \Omega_{c/2}$, the only possible limit is zero:
 \begin{equation*}
   \lim \nrm{Z_n - m} = 0, \quad \text{a.s.} \qedhere
 \end{equation*}
\end{proof}

   \subsection{Proof of Lemma \ref{def:TaylorPhi} and Proposition \ref{lmm=vitesseQuadratique}}

\begin{proof}[Proof of Lemma \ref{def:TaylorPhi}] 
Consider, for $h \in \ball$, the function $f_h(t) = \Phi(m+th),$ defined for  $t \in [0,1].$ We have $f_h(0) = \Phi(m) = 0$ and $f_h(1) = \Phi(m+h).$ It is also clear that 
the first order derivative $f'_h(t)$ of function $f_h$ satisfies $f'_h(t) = \Gamma_{m+th}.$
Consequently, a Taylor expansion with integral remainder of $f_h$ about $t=0$ gives us
 \[ 
\Phi(m+h) =  \Phi(m) + \int_0^1 \Gamma_{m+th}(h) \ dt.
\]
By Lemma 5.7 in \cite{Chaud92}, there is a constant $M_A$ such that for all $t \in [0,1],$
\[
\nrm{\Gamma_{m+th} - \Gamma_m}_L \leq M_A \nrm{h}
\]
where $\nrm{.}_L$ is the usual norm for bounded linear operators. Since $\Phi(m)=0,$ one gets
\[
\nrm{\Phi(m+h)   - \Gamma_m(h)} \leq \sup_{t \in [0,1]} \nrm{\Gamma_{m+th} - \Gamma_m}_L  \nrm{h}  \leq M_A \nrm{h}^2,
\]
and this concludes the proof.
\end{proof}

\begin{proof}[Proof of Proposition \ref{lmm=vitesseQuadratique}]
The proof is composed of 5 steps.

\paragraph{Step 1 --- a spectral decomposition.}
Recall that $\Gamma_m$ is:
\begin{align}
  \notag
\Gamma_m
& = \EE { \frac{1}{\| X-m\|} \left(
    \idt - \frac{( X-m) \otimes  (X-m)}{\| X-m\|^2} \right)
    } \\
\notag
&= \EE{\nrm{X-m}^{-1}} \idt - \EE{  \frac{1}{\| X-m\|} \left(
     \frac{( X-m) \otimes  (X-m)}{\| X-m\|^2} \right)
     } \\
  \label{eq=decompositionGammaM}
  &= \EE{\nrm{X-m}^{-1}} \idt - \Delta_m.
\end{align}
Since $\Gamma_m$ is bounded and symmetric, it is self-adjoint. Moreover, the operator $\Delta_m$ defined by \eqref{eq=decompositionGammaM} is trace class: it is self-adjoint, non negative, and if $(e_j)$ is an orthonormal basis, 
\begin{align*}
\sum_j \scal{e_j,\Delta_m e_j} 
&= \sum_j \EE{ \frac{ \scal{X-m,e_j}^2}{ \nrm{X-m}^3 } } \\
&\leq \EE{ \frac{1}{\nrm{X-m}} } < \infty. 
\end{align*}

Therefore $\Delta_m$ is compact, and there is an increasing sequence of eigenvalues $(\lambda_j),$ with possible repetitions, and an orthonormal basis $(v_j)$ of eigenvectors in $H$ such that:
\begin{align*}
  \forall j \in \xN, \quad \Gamma_m v_j &= \lambda_j v_j, \\
  \sigma(\Gamma_m) &= \{\lambda_j, j\in \xN\} \cup \left\{ \EE{\nrm{X-m}^{-1}} \right\}, \\
  & \!\!\!\!\!\!\!\!\!\! \lambda_j \limite{}{j\to \infty} \EE{\nrm{X-m}^{-1}}.
\end{align*}
Moreover, thanks to (\ref{eq=bndsGammam}), the smallest eigenvalue $\lmin$ of $\Gamma_m$ is strictly positive. 
For simplicity of notation, we rewrite this decomposition as follows, 
\[
\Gamma_m x = \sum_{\lambda\in\Lambda} \lambda \scal{e_\lambda, x} e_\lambda, \quad x \in H,
\] 
where 
$\Lambda$ is the multiset $\{\lambda_j,j\in\xN\},$ that can  account for eigenspaces of dimension larger than $1.$

In the following, we will need the operators:
\begin{align*}
  \alpha_k &= \idt - \gamma_k \Gamma_m,
  &
  \beta_n  &= \alpha_n \alpha_{n-1} \cdots \alpha_1.
\end{align*}
Since $\Gamma_m$ is bounded, these operators are well defined. Introducing the sequence  of real functions, for $n \in \xN,$ 
\[
f_n(x) = \prod_{k=1}^n (1- \gamma_k x),
\]
we see that  
$f_n(\cdot)$ and $f_n^{-1}(\cdot)$ are well defined on $\sigma(\Gamma_m),$ provided $\gamma_n  \EE{\nrm{X-m}^{-1}}< 1,$ which we can assume  without loss of generality. 
Elementary analysis shows that there exist constants $c_1$, $C_2, C_3$ such that:
\begin{equation}
  \label{eq=asymptotiqueBetaN}
\begin{aligned}
 \forall x \in  \sigma(\Gamma_m), 
 \quad
 c_1 \exp\left( - s_n x \right) &\leq f_n(x) \leq C_2 \exp\left( -s_n x \right),  \\
 \abs{s_n  -  \frac{\cG}{1-\alpha} n^{1 - \alpha}} &\leq C_3,
 \end{aligned}
 \end{equation}
 where we recall that $s_n = \sum_{k=1}^{n} \gamma_k$, and $\gamma_k = \cG k^{-\alpha}$. 
Then each  operator $\beta_n$ can be also expressed  as follows:
\[
\beta_n x = \sum_{\lambda\in\Lambda} f_n(\lambda) \scal{e_\lambda,x} e_\lambda, \quad x \in H,
\]
their inverses are bounded operators, and satisfy: $\beta_n^{-1}x = \sum_{\lambda\in\Lambda} f_n^{-1}(\lambda) \scal{e_\lambda,x} e_\lambda.$ 

\paragraph{Step 2 --- Decomposition of the algorithm.}

Let us rewrite the algorithm in the following way 
\begin{align*}
Z_{n+1}&=Z_n+\gamma_n \xi_{n+1}-\gamma_n\Phi(Z_n)\\
&= Z_n+\gamma_n \xi_{n+1}-\gamma_n(\Gamma_m(Z_n-m)+\delta_n)
\end{align*}
where $\delta_n=\Phi(Z_n)-\Gamma_m(Z_n-m)$ is the difference between the gradient of $G$ and the gradient of its quadratic approximation. 
Therefore:
\begin{equation}
  \label{eq=decompositionI}
  \forall k,\quad
  Z_{k+1} - m = \alpha_k (Z_k - m)  + \gamma_k \xi_{k+1} - \gamma_k \delta_k
\end{equation}
Rewriting $\alpha_{n-1}\alpha_{n-2}\cdots \alpha_{k+1}$ as $\beta_{n-1}\beta_k^{-1}$, we get by induction, 
\begin{equation}
  \label{eq=decompositionII}
  Z_n-m=\beta_{n-1}(Z_1-m) + \beta_{n-1}M_n - \beta_{n-1}R_{n-1},
\end{equation}
where
\begin{align*}
  R_n &= \sum_{k=1}^{n-1}\gamma_k \beta_k^{-1}\delta_k 
&
M_n &= \sum_{k=1}^{n-1}\gamma_k \beta_k^{-1}\xi_{k+1}.
\end{align*}
The first two terms of \eqref{eq=decompositionII} are what we would get if $G$ was exactly quadratic: a deterministic gradient part going to $m$, and 
a noise part; $R_n$ is the error term. We will  look at each of these terms in turn. 

\paragraph{Step 3 --- The deterministic term.}

We want to bound $\beta_{n-1}(Z_1 - m)$. The asymptotic 
behaviour of $f_n$ in eq. \eqref{eq=asymptotiqueBetaN} implies that
\[
\nrm{\beta_{n-1}} \leq C_2 \exp\left( - s_n \lmin\right),
\]
where $\lmin>0$ is the smallest eigenvalue of  $\Gamma_m$. Therefore
\begin{equation}
  \label{eq=termeDeterministe}
  \EE{\nrm{\beta_{n-1}(Z_1 - m)}^2} 
  \leq C \exp\left( -2 n^{1-\alpha}\right) \EE{\nrm{Z_1 - m}^2}.
\end{equation}

\paragraph{Step 4 --- The martingale.}

The fact that the $\beta_k$ are operators (instead of real numbers) makes matters more complicated. To deal with this problem, we use the  spectral decomposition of the sequence of self-adjoint operators $(\beta_k).$ 
\newcommand{\Ml}{M^\lambda}

More precisely, we decompose $M_n = \sum_{\lambda\in\Lambda} \scal{e_\lambda,M_n} e_\lambda = \sum_\lambda \Ml_n e_\lambda$.
For each $\lambda \in \Lambda$, $\Ml_n$ is a martingale, and
\begin{align*}
  \E[(\Ml_n)^2] 
  &= \sum_{k\leq n-1} \gamma_k^2 f^{-2}_k(\lambda) \E\left[
  \scal{\xi_{k+1},e_\lambda}^2 \middle| \CF_k \right],
\end{align*}
since $\E\left[
  \scal{\xi_{k'},e_\lambda}\scal{\xi_{k+1},e_\lambda}  \middle| \CF_k \right] = 0$ when $k'<k+1.$ 
Summing now over $\lambda \in \Lambda$, we get:
\begin{align*}
  \E\left[\nrm{\beta_{n-1}M_n}^2\right]
  &= \sum_\lambda f_{n-1}^2(\lambda) \EE{(\Ml_n)^2}, \\
  %
  &\leq \sum_\lambda \sum_{k\leq n-1} \gamma_k^2 \left(
      \frac{f_{n-1}(\lambda)}{f_k(\lambda)}
      \right)^2 \EE{\scal{\xi_{k+1},e_\lambda}^2} .
\end{align*}
However, for any $k,n$, and any $\lambda \in \Lambda$,  
\[
  \frac{f_{n-1}(\lambda)}{f_k(\lambda)}
  = \prod_{j=k+1}^{n-1} \left( 1 - \lambda\gamma_j\right) 
  \leq \frac{f_{n-1}(\lmin)}{f_k(\lmin)}.
  \]
  This uniformity in $\lambda$ allows us to reconstruct $\EE{\nrm{\xi_{k+1}}^2}$, which is bounded by $1$, thanks to \eqref{eq=crochetXi}.  We obtain:
\begin{align*}
  \notag
  \E\left[\nrm{\beta_{n-1}M_n}^2\right]
  &\leq \sum_{k\leq n-1} \gamma_k^2 \left(
      \frac{f_{n-1}(\lmin)}{f_k(\lmin)}
      \right)^2 \sum_\lambda \EE{\scal{\xi_{k+1},e_\lambda}^2} ,\\
  \notag
    &\leq \sum_{k\leq n-1} \gamma_k^2 \left(
      \frac{f_{n-1}(\lmin)}{f_k(\lmin)}
      \right)^2 \EE{\nrm{\xi_{k+1}}^2}, \\
  \notag
    &\leq \sum_{k\leq n-1} \gamma_k^2 \left(
      \frac{f_{n-1}(\lmin)}{f_k(\lmin)}
    \right)^2. 
\end{align*}
Now we use the bounds \eqref{eq=asymptotiqueBetaN} on $\beta_n$: 
\begin{align}
  \notag
  \E\left[\nrm{\beta_{n-1}M_n}^2\right]
    & \leq \frac{C_2^2}{c_1^2} \sum_{k\leq n-1} \gamma_k^2
      \exp\left(- \sum_{j=k+1}^n \gamma_j \right) \\
      \label{eq=belleSomme}
    &\leq C \sum_{k\leq n-1} \gamma_k^2 \exp\left( - \frac{1}{1-\alpha} \left( n^{1 - \alpha} - k^{1- \alpha} \right) \right). 
\end{align}
The exponential terms are very small when $k$ is much smaller than $n$, therefore we isolate the last terms.  To do that, we choose $l(n)$ such that, 
\begin{equation}
  \label{eq=choixDeL}
l(n)^{1-\alpha} = n^{1 - \alpha} - c_\alpha\ln(n) \ , 
\end{equation}
with $c_\alpha$ to be chosen later. 
The first part of the sum \eqref{eq=belleSomme} (for $k\leq l(n)$) gives us:
\begin{align}
  \notag
  \sum_{k\leq l(n)} \gamma_k^2 \exp\left(
     - \frac{1}{1-\alpha} \left( n^{1 - \alpha} - k^{1- \alpha} \right) \right) 
  &\leq 
  \sum_{k\leq l(n)} \gamma_k^2 \exp\left(
     - \frac{c_\alpha}{1-\alpha} \ln(n) \right) \\
     \label{eq=belleSommeDebut}
  &\leq
  \cG^2\sum_{k\leq n}  k^{-2 \alpha - \frac{c_\alpha}{1-\alpha}}.
\end{align}
This can be made smaller than any prescribed inverse power of $n$, if we choose $c_\alpha$ large enough. In the second part of the sum \eqref{eq=belleSomme}, for $k>l(n)$, we bound the exponential by $1$ and $\gamma_k$ by $\gamma_{l(n)}$: 
\begin{align*}
  \sum_{k > l(n)} \gamma_k^2 \exp\left(
     - \frac{1}{1-\alpha} \left( n^{1 - \alpha} - k^{1- \alpha} \right) \right) 
     &\leq (n-l(n)) \gamma_{l(n)}^2.
 \end{align*}
 The number of terms $n- l(n)$ is equivalent to $\frac{c_\alpha}{1-\alpha} \ln(n)n^\alpha$, and $\gamma_{l(n)}\sim \cG n^{-\alpha}$. Therefore, the whole second term is equivalent to $ c \ln(n) n^{-\alpha},$ where $c$ depends on $c_\alpha$ and $c_\gamma.$ For $c_\alpha$ large enough, this dominates the first term \eqref{eq=belleSommeDebut}. Finally we get:
\begin{eqnarray}
\label{eq=termeMarting}
 \E\left[ \nrm{ \beta_{n-1} M_n}^2 \right] &\leq & C \frac{\ln(n)}{n^\alpha}.
 \end{eqnarray}

\paragraph{Step 5 --- the error term and the conclusion.}

The error term is $\beta_{n-1} \sum_{k=1}^n \gamma_k \beta_k^{-1}\delta_k$, where $\delta_k = \Phi(Z_k) - \Gamma_m(Z_k -m)$.  With Lemma~\ref{def:TaylorPhi}, we get that
\begin{equation}
  \exists r,C_r \quad
  \forall k, \nrm{Z_k - m} \leq r \implies \nrm{\delta_k} \leq C_r \nrm{Z_k - m}^2. 
  \label{eq=deltaK}
\end{equation}
Since $Z_n$ converges a.s. to $m$, we deduce two things about $\delta_k$: it is almost surely bounded, and 
\eqref{eq=deltaK} becomes a.s. eventually true. To use these facts we introduce the following sequence of events:
\[
  \Omega_N 
  = \left\{ \omega,
  \begin{array}{r}
    \forall n \geq N, \forall k\geq n - l(n), \quad  \nrm{Z_k(\omega) - m} \leq 1/K \\
    \qquad \text{ and } \nrm{\delta_k(\omega)} \leq C_r\nrm{Z_k(\omega) - m}^2\\
    \multicolumn{1}{l}{
    \forall k, \nrm{\delta_k(\omega)} \leq N.
    }
  \end{array}
  \right\},
\]
for a value of $K$ to be chosen later, and $l(n)$ defined by \eqref{eq=choixDeL}. This sequence is increasing and $\bigcup \Omega_N  = \Omega$; from now on we work on $\Omega_N$. 

Once more, since $\beta_{n-1}\beta_k^{-1}$ is very small when $k$ is much smaller than $n$, only the last terms in the sum defining $R_n$ matter. 
This is why we re-use the definition of $l(n)$ and cut the sum in two parts. For $\omega \in \Omega_N$, and $n\geq N$, 
\begin{align*}
   \nrm{\beta_{n-1}R_n}^2
  &\leq \left(
    \sum_{k=1}^n \gamma_k \nrm{\beta_{n-1}\beta_k^{-1}} \nrm{\delta_k}
    \right)^2 \\
  &\leq 2 N^2 \left(
      \sum_{k=1}^{l(n)} \gamma_k \nrm{\beta_{n-1}\beta_k^{-1}}
      \right)^2
      + 2 C_r^2 \left( \sum_{k=l(n) + 1}^n \gamma_k \nrm{Z_k - m}^2\right)^2 \\
  &\leq 2 N^2 \left(
      \sum_{k=1}^{l(n)} \gamma_k \nrm{\beta_{n-1}\beta_k^{-1}}
      \right)^2
      + 2\frac{ C_r^2}{K^2}(n-l(n))\gamma_{l(n)}  \sum_{k=l(n) + 1}^n \gamma_k \nrm{Z_k - m}^2.
\end{align*}
where we used the crude bound $\nrm{\delta_k}\leq N$ in the first part, and for the second part,  $\nrm{\beta_{n-1} \beta_k^{-1}}\leq 1$  and the definition of $\Omega_N$.

As before, it is easy to see that the first term is bounded by any prescribed
inverse power of $n$, say $n^{-42}$.  
For the second term, 
we already know that $(n-l(n))\gamma_{l(n)}$ is bounded. Therefore, on $\Omega_N$ and  for $n\geq N,$
 
\begin{equation}
  \label{eq=bidule}
  \nrm{\beta_{n-1}R_n}^2
  \leq 
  \frac{CN^2}{n^{42}} + \frac{C}{K^2} 
        \sum_{k=l(n) +1}^n \gamma_k \nrm{Z_k - m}^2.
\end{equation}

Combining now \eqref{eq=decompositionII}, \eqref{eq=termeDeterministe}, \eqref{eq=termeMarting}
and  \eqref{eq=bidule}, we get, for $n\geq N$ and some new  constant $C$ 
\begin{align*}
  \EE{\ind{\Omega_N} \nrm{Z_n - m}^2 }  
  &\leq \frac{C \ln(n)}{n^\alpha} + \frac{C}{K^2} \sum_{k=l(n)+1}^n \gamma_k \EE{\ind{\Omega_N} \nrm{Z_k -m}^2} \\
  &\leq \frac{C \ln(n)}{n^\alpha} + \frac{C'}{K^2} \sup_{l(n)<k\leq n} \EE{\ind{\Omega_N} \nrm{Z_k -m}^2}.
\end{align*}
 
Let us choose $K$ such that $K^2 \geq 2C'$. Then
\[
\forall n\geq N, \quad   \EE{\ind{\Omega_N} \nrm{Z_n - m}^2 } \leq \frac{C \ln(n)}{n^\alpha} + \frac{1}{2} \max_{l(n)< k \leq n} \EE{\ind{\Omega_N} \nrm{Z_k - m}^2 }.
\]
Defining $u_n = \EE{\ind{\Omega_N} \nrm{Z_n - m}^2}$, this reads:
\begin{equation}
  \label{eq:induction}
\forall n\geq N, \quad  u_n \leq \frac{C \ln(n)}{n^\alpha} + \frac{1}{2} \max_{l(n)< k \leq n} u_k.
\end{equation}
Let us prove by induction that, for some $N'$ large enough,
and for $C'' = 4C$, 
\[ \forall n\geq N', \quad  u_n \leq \frac{C''\ln(n)}{n^\alpha}. \]
Suppose that $u_k\leq \frac{C''\ln(k)}{k^\alpha}$ 
for all $k\leq n$, and let us prove that $u_{n+1} \leq \frac{C'' \ln(n+1)}{(n+1)^\alpha}$. Using \eqref{eq:induction}, we know that:
\begin{align*}
  u_{n+1} 
  &\leq \frac{C \ln(n+1)}{(n+1)^\alpha} + \frac{1}{2} \max_{l(n+1) < k \leq n+1} u_k \\
  &\leq \frac{C \ln(n+1)}{(n+1)^\alpha} + \frac{1}{2} \max\left(u_{n+1}, \frac{C'' \ln(l(n+1))}{l(n+1)^\alpha}\right) 
\end{align*}
If the max on the right hand side is $u_{n+1}$, we get:
\[ \frac{1}{2} u_{n+1} \leq \frac{C \ln(n+1)}{(n+1)^\alpha}\]
which is the desired result since  $2C \leq 4C = C''$.  
If this is not the case, then the $\max$ is $\frac{C'' \ln(l(n+1))}{l(n+1)^\alpha}$. However, $l(n+1) \sim (n+1)$ 
so for $n$ larger than some $N'$, $\frac{\ln(l(n+1))}{l(n+1)^\alpha} \leq \frac{3}{2} \frac{\ln(n+1)}{(n+1)^\alpha} $. 
Hence
\begin{align*}
  u_{n+1} &\leq 
  \frac{C \ln(n+1)}{(n+1)^\alpha} + \frac{3}{4} \frac{C'' \ln(n+1)}{(n+1)^\alpha}  \\
  &\leq 4C \frac{\ln(n+1)}{(n+1)^\alpha}  
  =  C''   \frac{\ln(n+1)}{(n+1)^\alpha}.
 \end{align*}
This concludes the induction step and the  proof of Proposition \ref{lmm=vitesseQuadratique}. 
\end{proof}

\subsection{Proof of Theorem \ref{prop:vitesse}}
  We use the same decomposition as in \cite{Pel00}. It consists in linearizing the target function $\Phi$ around the true value $m.$ Recall the following decomposition of the error \eqref{eq=decompositionI},
\[
  \forall k,\quad
  Z_{k+1} - m = (\idt - \gamma_k \Gamma_m) (Z_k - m)  + \gamma_k \xi_{k+1} - \gamma_k \delta_k,
\]
where $\xi_k$ is a martingale difference sequence and $\delta_k$ are error terms, $\delta_k =  \Phi(Z_k)-\Gamma_m(Z_k-m)$. 
Defining now,
\[
T_n \egaldef Z_n - m, \quad 
\Tbar_n\egaldef \Zbar_n - m 
\quad \text{and} \quad
\newM_{n+1}\egaldef \sum_{k=1}^n\xi_{k+1},
\]
and rearranging the previous expression, we obtain:
\[
  \Gamma_m T_k  = \xi_{k+1} - \delta_k + \frac{1}{\gamma_k}\left(T_k - T_{k+1}\right).
\]
Summing these equalities, it comes,
\[
n\Gamma_m\Tbar_n = 
\sum_{k=1}^n \frac{1}{\gamma_k}\left( T_k - T_{k+1} \right)
- \sum_{k=1}^n\delta_k
+  \newM_{n+1}.
\]
Applying  Abel's transform, and dividing by $\sqrt{n}$ yields:
\[
\sqrt{n}\Gamma_m\Tbar_n = 
\frac{1}{\sqrt{n}}\left(
\frac{T_1}{\gamma_{1}}
- \frac{T_{n+1}}{\gamma_{n}}
+ \sum_{k=2}^n T_k\left[
    \frac{1}{\gamma_k}-\frac{1}{\gamma_{k-1}}
  \right]
- \sum_{k=1}^n\delta_k
\right)
+ \frac{1}{\sqrt{n}} \newM_{n+1}.
\]
To prove that last term is a martingale for which the CLT holds,
\[\frac{\newM_n}{\sqrt{n}} \cvl \mathcal{N}\left(0, \covLimite\right),\]
we need  to check that the assumptions of Theorem 5.1 in \citep{Jak88} are fulfilled. We first have  that the martingale difference sequence is \textit{a.s.} bounded, $\forall n \nrm{\xi_n}\leq 2.$ 
Let us define
\begin{equation*}
\covLimite_n = \EFn{\xi_{n+1}\otimes \xi_{n+1} },
\end{equation*}
which can also be decomposed as follows
\begin{equation*}
\covLimite_n = \EFn{ \frac{(X-Z_n)}{\nrm{X-Z_n}} \otimes \frac{(X-Z_n)}{\nrm{X-Z_n}}} - \Phi(Z_n)\otimes\Phi(Z_n).
\end{equation*}
Since $\Phi(m)=0,$  we have by a direct computation,
\[\nrm{\Phi(Z_n)} \leq \EE{\frac{2}{\nrm{X-m}}} \nrm{Z_n - m}.
\]
Using now, 
for $(a,b) \in H \times H,$ the inequality $\nrm{a\otimes b}_L \leq \nrm{a} \nrm{b},$ where $\nrm{a\otimes b}_L$ is the usual the norm for linear operators, we directly get, with Theorem \ref{prop:convH},
\[
\nrm{\Phi(Z_n)\otimes\Phi(Z_n)}_L \rightarrow 0, \quad a.s.
\]
With similar arguments, it is easy to show that
\begin{align*}
\nrm{\covLimite - \EFn{ \frac{(X-Z_n)}{\nrm{X-Z_n}} \otimes \frac{(X-Z_n)}{\nrm{X-Z_n}}} }_L & \leq  2 \EFn{\nrm{\frac{(X-Z_n)}{\nrm{X-Z_n}}  - \frac{(X-m)}{\nrm{X-m}}}} \\
 & \leq 4 \EE{\frac{1}{\nrm{X-m}}} \nrm{Z_n - m},
\end{align*}
so that $\nrm{\covLimite_n - \covLimite}_L \rightarrow 0 \ a.s.,$  when $n$ tends to infinity.
 Then condition 5.2 in \citep{Jak88} is satisfied and is a consequence of a direct application of Chow's Lemma, see for instance \cite[page 22]{Duf97}.

 Now, it remains to prove that
\begin{equation}
  \label{eq=tripleReste}
  \frac{1}{\sqrt{n}}\left(
    \frac{T_{n+1}}{\gamma_{n}} 
    - \sum_{k=2}^n T_k \left[\frac{1}{\gamma_k}-\frac{1}{\gamma_{k-1}}\right]
    + \sum_{k=1}^n \delta_k 
  \right)
  \cvp 0.
\end{equation}
Let us denote by $A_n$, $A'_n$ and $A''_n$ the three terms.

Recall that $\EE{\ind{\Omega_N} \nrm{T_n}^2} \leq C_N \frac{\ln(n)}{n^\alpha},$ thanks to Proposition \ref{lmm=vitesseQuadratique}. 
For the first term $A_n = \frac{T_{n+1}}{\sqrt{n}\gamma_n}$, we have:
\begin{align*}
  \EE{\ind{\Omega_N} \nrm{A_n}^2}
  &\leq C'_N n^{2\alpha - 1} \frac{\ln(n)^2}{n^{2\alpha}} = \frac{C'_N \ln(n)^2}{n}.
\end{align*}
Therefore $A_n \cvp 0$. 

Let us turn to the second term $A'_n$. Since 
\( \gamma_k^{-1} - \gamma_{k-1}^{-1} \leq 2\alpha c_{\gamma}^{-1}  k^{\alpha - 1}, \)
we have, for two positive constants $C_0, C_1$, 
\begin{align*}
  \EE{\nrm{A'_n}\ind{\Omega_N}}
  &\leq \frac{2\alpha c_{\gamma}^{-1}}{\sqrt{n}} \sum_{k\leq n} \EE{\ind{\Omega_N}\nrm{T_k}} k^{\alpha - 1} \\
  &\leq \frac{C_0}{\sqrt{n}}\sum_{k \leq n} \sqrt{\ln(k)} k^{\alpha/2 - 1} \\
  &\leq C_1\sqrt{\ln(n)} n^{\alpha/2 - 1/2},
\end{align*}
which goes to zero since $\alpha < 1$. Therefore $A'_n\cvp 0$. 

Finally, for the last term $A''_n$, since there exists a positive constant $C_2$ such that $\nrm{\delta_k} \leq C_2 \nrm{Z_k - m}^2$, we have:
\begin{align*}
  \EE{\ind{\Omega_N} \nrm{A''_n}} 
  &\leq \frac{1}{\sqrt{n}}\sum_{k \leq n} \EE{\ind{\Omega_N} \nrm{T_k}^2} \\
  &\leq \frac{C_N}{\sqrt{n}} \sum_{k \leq n} \ln(k)k^{-\alpha}. 
\end{align*}
Since the right hand side term converges to zero (as can be seen e.g. by Kronecker's lemma, using the fact that $\alpha > 1/2$), $C_n \cvp 0$, therefore \eqref{eq=tripleReste} holds, and Theorem~\ref{prop:vitesse} is finally proved.

\medskip

\noindent \textbf{Acknowledgements.} We would like to thank the referees for their helpful and valuable suggestions. We also thank the company M\'ediam\'etrie for allowing us to illustrate our methodologies with their data.

\bibliographystyle{apalike}
\bibliography{biblio_mediane}

\end{document}